\documentclass[a4paper,twoside,12pt]{article}

\usepackage[all]{xy}
\usepackage{xspace}
\usepackage{amsthm,tikz,hyperref,color,amsmath,graphicx,amsfonts,gensymb,mathdots,geometry,amssymb,tensor,enumitem,wasysym,makeidx,stmaryrd}
\usepackage[utf8]{inputenc}
\usepackage[T1]{fontenc}
\usepackage[english]{babel}
\usetikzlibrary{cd,arrows}

\def\Trace{ \operatorname{Trace}}

\def\F{\mathbb F}
\def\Fqbar{\overline{\F}_q}
\def\Ql{\overline{\mathbb Q}_{\ell}}

\def\G{\mathbf G}
\def\T{\mathbf T}

\def\U{\mathbf U}
\def\P{\mathbf P}

\def\L{\mathbf L}
\def\M{\mathbf M}

\def\Y{\mathbf Y}

\def\Go{\mathbf G^{\circ}}
\def\To{\mathbf T^{\circ}}

\def\Po{\mathbf P^{\circ}}
\def\Qo{\mathbf Q^{\circ}}
\def\Lo{\mathbf L^{\circ}}
\def\Mo{\mathbf M^{\circ}}

\def\GoF{\mathbf G^{\circ F}}

\def\LoF{\mathbf L^{\circ F}}
\def\MoF{\mathbf M^{\circ F}}

\def\Res{\operatorname{Res}}
\def\Ind{\operatorname{Ind}}

\def\ad{\operatorname{ad}}

\newtheoremstyle{tdp}{}{}{\itshape}{}{\bfseries}{:}{.5em}{}
\newtheoremstyle{ddp}{}{}{}{}{\bfseries}{:}{.5em}{}
\newtheoremstyle{not}{}{}{}{}{\itshape}{:}{.5em}{}

\theoremstyle{tdp}
\newtheorem{definition}{Definition}[section]
\newtheorem{theorem}[definition]{Theorem}
\newtheorem{prop}[definition]{Proposition}
\newtheorem{corollary}[definition]{Corollary}
\newtheorem{lemma}[definition]{Lemma}

\theoremstyle{ddp}

\theoremstyle{not}
\newtheorem*{remark}{Remark}

\title{Mackey formula for disconnected reductive groups.}

\author{Sergio Cía}

\begin{document}
 
\maketitle

\begin{abstract}
We prove a Mackey formula for representations of finite groups of Lie type, in the case where the groups come from disconnected reductive groups.
\end{abstract}

\section{Introduction.}

The Deligne-Lusztig theory is an important tool for the study of representations of finite reductive groups, based on the underlying geometry of the algebraic groups from where these groups come. It  
let us, for example, organize and assemble the irreducible representations of these groups, and calculate some of the characters of these representations. 

This theory relies heavily on the fact that the algebraic groups are connected, and one may be interested in having a similar theory for disconnected groups. 
Digne and Michel succeeded in adapting the Deligne-Lusztig theory to disconnected reductive groups in \cite{DiMi}, 
and they showed that it behaves as well as the theory for connected groups.

The example in which we will focus on is the Mackey formula for the Deligne-Lusztig maps: let $q$ be a power of a prime number, let $\Go$ be a connected reductive group over $\Fqbar$ and let $F$ be a Frobenius endomorphism endowing $\Go$ with an $\F_q$-structure. Let $\Lo$, $\Mo$ be $F$-stable Levi subgroups of $\Go$. Then, if $R_{\Mo}^{\Go}$ is the Lusztig induction from $\Mo$ to $\Go$ and $\tensor*[^*]{}{}R_{\Lo}^{\Go}$ is the Lusztig restriction from $\Go$ to $\Lo$, then, under some mild conditions (see \cite[theorem 3.9]{BonMi}), we have 
\[
\tensor*[^*]{}{}R_{\Lo}^{\Go}\circ R_{\Mo}^{\Go}=\sum_{x\in \LoF\backslash \mathcal S_{\Go}(\Lo,\Mo)^F/ \MoF} R_{\Lo\cap\tensor*[^x]{}{}\Mo}^{\Lo} \circ \tensor*[^*]{}{}R_{\Lo\cap\tensor*[^x]{}{}\Mo}^{\tensor*[^x]{}{}\Mo} \circ \ad(x),
\]
where $\mathcal S_{\Go}(\Lo,\Mo)$ is the set of $x\in\Go$ such that $\Lo\cap\tensor*[^x]{}{}\Mo$ contains a maximal torus of $\Go$.

The purpose of this paper is to generalize this formula to a possibly disconnected reductive group $\G$, having $\Go$ as identity component:

\begin{theorem}
Let us assume that $\G/\Go$ consists on semisimple elements and that it has "enough normal subgroups" (see Corollary \ref{main_corollary}). 
Let $\L$ and $\M$ be Levi subgroups of $\G$ (in a sense to be precised in \textsection 2). Then the Mackey formula holds under the same conditions as in \cite[theorem 3.9]{BonMi}
\end{theorem}


Digne and Michel proved the result when both $\L$ and $\M$ are tori of $\G$ (see \cite[theorem 3.2]{DiMi3}). Moreover, they obtained 
a formula for each of the connected components of $\G$, in the case where $\L$ and $\M$ are contained in $F$-stable parabolic subgroups of $\G$ (see \cite[theorem 3.2]{DiMi}). We used their formula for tori to guess our general Mackey formula: rewriting the formula, one can obtain the set of indices 
$\mathcal S_{\G}(\L,\M)$, and find which induction and restriction maps we should use in the sum. Then, the strategy of the proof follows the argument of Bonnafé in the paper \cite{Bon} by induction reducing to the unipotent part of centralizers of semisimple elements (see \cite[proposition 2.3.4]{Bon}). This allows us to reduce the formula to the connected case.

In \textsection 2, we will give the general setup for the article, bringing some results from \cite{DiMi}. In \textsection 3, we will define the induction and restriction maps that will appear in the formula. We will get a character formula for them and obtain the result that let us reduce the problem to connected reductive groups. Finally, in \textsection 4, we will state the Mackey formula, and we will prove it.

I would like to thank Olivier Dudas for all the help given during the discussions, and Jean Michel for answering my questions on his article with François Digne.

\section{General setting.}
\label{sec_gen_setting}

%
%

We will consider $\G$ a reductive group (not necessarily connected) over $\Fqbar$. Let $\Go$ be its identity component. Let $F$ be a Frobenius endomorphism of $\G$, endowing it with a rational structure over $\F_q$. Let us denote the same way its restriction to $\Go$, and let $\G^F$, $\GoF$ be the corresponding groups of rational points.

We take the definitions 
from \cite{DiMi}: We call \emph{parabolic subgroup} of $\G$ a group of the form $\P=N_{\G}(\Po)$, where $\Po$ is a parabolic subgroup of $\Go$; and \emph{Levi subgroup} of $\G$ one of the form $\L=N_{\G}(\Lo,\Po)$ for a parabolic subgroup $\Po$ and a Levi complement  $\Lo$ of $\Po$. If $\U=\operatorname{R}_u(\Po)$ we still have a Levi decomposition of $\P$ in $\G$: $\P=\L\ltimes\U$. 
 
 
We will suppose that the quotient $\G/\Go$ 
consists on semisimple elements.
This condition ensures that 
an element is "quasi-semisimple" if and only if it is semisimple (see \cite[definition 1.1]{DiMi} for a definition of quasi-semisimple element). 
 For a semisimple element $s$ of $\G$ we denote $\G(s)$ its centraliser. Again, the previous condition on $\G$ implies that all the unipotent elements of $\G(s)$ 
lie in its identity component $\G(s)^{\circ}$ (see \cite[remark 2.7]{DiMi}).

We have that $\G(s)^{\circ}=(\G^{\circ}(s))^{\circ}$: The second set is included in the first set; and $\G(s)^{\circ}$ is connected, then $\G(s)^{\circ}\subset \Go$ . Therefore, we get $\G(s)^{\circ}\subset \Go(s)$ and the other inclusion. This allows us to apply the results on the article by Digne and Michel to $\G(s)^{\circ}$: Following \cite[theorem 1.8]{DiMi}, $\G(s)^{\circ}$ is a connected reductive group. \cite[proposition 1.11]{DiMi} tells us that if $s\in\L\subset\P$, then $\P(s)^{\circ}$ is a parabolic subgroup of $\G(s)^{\circ}$, $\U(s)^{\circ}$ is its unipotent radical (where $\U=\operatorname R_{u}(\Po)$) and $\L(s)^{\circ}$ is a Levi subgroup of $\G(s)^{\circ}$.

%
%

\section{Green functions over $\G$ and generalised character formula.}

We will always consider representations over $\Ql$, with $\ell$ a prime number coprime with $q$. 
%
%
Let us recall the character formula for the maps $R_{\L}^{\G}$ and $\tensor*[^*]{}{}R_{\L}^{\G}$ defined in \cite[proposition 2.6]{DiMi}: Let $\L$ be a rational Levi subgroup of $\G$. 
We take $g\in \G^F$, $l\in\L^F$, $\chi$  a class function on $\G^F$ and $\psi$ a class function on $\L^F$. Let $Q_{\Lo}^{\Go}$ be the map from $\Go_u\times \Lo_u$ to $\Ql$ that sends $(u,v)$ to $\operatorname{Trace}((u,v)|H_c^*(\Y^{\circ}_{\U}))$, where $\Y^{\circ}_{\U}$ is the Deligne-Lusztig variety that defines $R_{\Lo}^{\Go}$. $Q_{\Lo}^{\Go}$ is called the \emph{Green function} associated to $\Go$ and $\Lo$. Then:
\begin{itemize}
\item If $\chi$ is a class function of $\L^F$, and $g=su$ is the Jordan decomposition of $g$, with $s$ semisimple and $u$ unipotent, then
\[
R_{\L}^{\G}(\chi)(g) = \frac{1}{|\L^F||\G(s)^{\circ F}|}\sum_{\{h\in\G^F|s\in\tensor*[^h]{}{}\L\}}\sum_{v\in (\tensor*[^h]{}{}\L(s)^{\circ})_u^F} Q_{\tensor*[^h]{}{}\L(s)^{\circ}}^{\G(s)^{\circ}}(u,v^{-1})\tensor*[^h]{}{}\chi(sv).
\]

\item If $\phi$ is a class function of $\G^F$ and $l=tv$ is the Jordan decomposition of $l$, then:
\[
\tensor*[^*]{}{}R_{\L}^{\G}(\psi)(l)=\frac{1}{|\G(t)^{\circ F}|} \sum_{u\in (\G(t)^{\circ})_u^F} Q_{\L(t)^{\circ}}^{\G(t)^{\circ}}(u^{-1},v)\psi(tu).
\]
\end{itemize}

We will need a character formula for a more general Lusztig induction and restriction:
Let $\G'$ be a normal subgroup of $\G$ containing $\Go$. Since the unipotent radical of $\G'$ is contained in that of $\Go$ (it is connected, unipotent, closed and normal in $\Go$ too), we have that $\G'$ is a reductive group. Let $\L'$ be a rational Levi subgroup of $\G'$.
We note
\[
R_{\L'}^{\G}= \Ind_{\G'^F}^{\G^F}\circ R_{\L'}^{\G'}, \tensor*[^*]{}{}R_{\L'}^{\G}= \tensor*[^*]{}{}R_{\L'}^{\G'} \circ \Res_{\G'^F}^{\G^F}.
\]
We take $g\in\G^F$ and $\chi$ a class function of $\L'^F$. 
Since $\G'$ is a normal subgroup of $\G$, we have that 
\begin{equation} \label{eq_R_ind}
R_{\L'}^{\G}\chi(g)=\left\{\begin{array}{ll}
\sum_{\tau\in\G^F/\G'^F}R_{\L'}^{\G'}\tensor*[^{\tau}]{}{}\chi(g)& \text{if } g\in\G'^F,\\
0 & \text{otherwise.}
\end{array}\right. 
\end{equation} 


Let $g\in \G'^F$, and let $g=su$ be the Jordan decomposition of $g$. We remark that $s$ is in $\G'^F$, since $u$ is in $\Go$.

\begin{prop}[Generalized character formula]\label{char}
Let $g\in\G^F$ and $g=su$ be the Jordan decomposition of $g$. We have that
\[
R_{\L'}^{\G}\chi(g)= 
\frac{1}{|\L'^F||\G(s)^{\circ F}|}
\sum_{\{h\in\G^F|s\in\tensor*[^h]{}{}\L'\}}
\sum_{v\in (\tensor*[^h]{}{}\L'(s)^{\circ})_u^F}
Q_{\tensor*[^h]{}{}\L'(s)^{\circ}}^{\G(s)^{\circ}}(u,v^{-1})
\tensor*[^h]{}{}\chi(sv).
\]
\end{prop}
\begin{remark}

\begin{itemize}[label=$\bullet$]

\item Since all unipotent elements of $\G$ lie in $\Go$ and $\G'$ contains $\Go$, we have that $g$ belongs to $\G'^F$ if and only if $s$ belongs to $\G'^F$. Therefore, the formula holds if we take $g$ not in $\G'^F$, since 
then the first sum on the right hand side is zero. 
\item The formula for the modified restriction is the same as in the usual case, it consists on just applying the usual character formula to an element of $\L'^F$. 
\end{itemize}
\end{remark}

\begin{proof}
Following the previous remark, we will suppose that $g$ belongs to $\G'^F$. 

From now on, we fix a set of representatives $\{\dot{\tau}\}_{\tau\in\G^F/\G'^F}$ of the left cosets of $\G^F$ modulo $\G'^F$. We recall that $\
\tensor*[^{\tau}]{}{}\chi(g)=\chi(\dot{\tau}^{-1}g\dot{\tau})$. From now on, we fix a left coset $\tau$, and we call $\phi$ the automorphism given by conjugation by $\dot{\tau}^{-1}$. 

We have that $\phi(g)=\phi(s)\phi(u)$, and $\phi(s)$ and $\phi(u)$ are still semisimple and unipotent elements, respectively. Therefore, we will apply the character formula to each of the terms of the sum in \ref{eq_R_ind}, that is, we will apply it to each element $\phi(g)=\phi(s)\phi(u)$. 
We will do some simplifications before:
\begin{itemize}[label=$\bullet$]
\item Let $h\in\G^F$. We have that $s$ belongs to $\tensor*[^h]{}{}\L'$ if and only if $\phi(s)$ belongs to $\phi(\tensor*[^h]{}{}\L')=\tensor*[^{\dot{\tau}^{-1}h}]{}{}\L'$. Therefore $\{h'\in\G^F|\phi(s)\in\tensor*[^{h'}]{}{}\L'\}=
\dot{\tau}^{-1}\{h\in\G^F|s\in\tensor*[^h]{}{}\L'\}$. And the first sum will have the elements $h'$ in $\G'^F$, so they must come from elements $h$ in $\dot{\tau}\G'^F=\tau$. Therefore
\[
\{h\in\G'^F|\phi(s)\in\tensor*[^h]{}{}\L'\}=
\dot{\tau}^{-1}\{h\in\tau|s\in\tensor*[^h]{}{}\L'\}.
\] The first sum will have as indexes the $h$ in $\tau$; in exchange, we will have $\dot{\tau}^{-1} h$ appearing in the terms of the sum. 
\item The algebraic group $\tensor*[^{\dot{\tau}^{-1}h}]{}{}\L'(\phi(s))^{\circ}$ will appear in the second set of indexes and in the sum. But:
\[ 
\begin{array}{rcl}
\phi(\tensor*[^{h}]{}{}\L'(s)^{\circ}) & = & \{\phi(a)~|~a\in\tensor*[^{h}]{}{}\L', as=sa\} \\
 & = & \{b\in\phi(\tensor*[^{h}]{}{}\L')~|~b\phi(s)=\phi(s)b\} \\
 & = & \tensor*[^{\dot{\tau}^{-1}h}]{}{}\L'(\phi(s))^{\circ},
\end{array}
\] since $a$ commutes with $s$ if and only if $b=\phi(a)$ commutes with $\phi(s)$. Therefore, we are allowed to take $v$ in $\tensor*[^{h}]{}{}\L'(s)^{\circ}$ in the second sum, and the set $\phi(\tensor*[^{h}]{}{}\L'(s)^{\circ})$ and the element $\phi(v)$ will appear in the terms of the sum.
\end{itemize}

So we have that
\begin{align*}
R_{\L'}^{\G}\chi(g) = & 
\frac{1}{|\L'^F||\G'(s)^{\circ F}|}
\sum_{\tau\in\G^F/\G'^F} \\
& \sum_{\{h\in\tau|s\in\tensor*[^h]{}{}\L'\}}
\sum_{v\in (\tensor*[^h]{}{}\L'(s)^{\circ})_u^F}
Q_{\phi(\tensor*[^h]{}{}\L'(s)^{\circ})}^{\phi(\G'(s)^{\circ})}(\phi(u),\phi(v)^{-1})
\tensor*[^{\dot{\tau}^{-1}h}]{}{}\chi(\phi(sv)).
\end{align*} 

First, we will rewrite the first term of the sum, the Green function applied to $(\phi(u), \phi(v))$: 

\begin{lemma}
\[Q_{\phi(\tensor*[^h]{}{}\L'(s)^{\circ})}^{\phi(\G'(s)^{\circ})}(\phi(u),\phi(v)^{-1})= Q_{\tensor*[^h]{}{}\L'(s)^{\circ}}^{\G'(s)^{\circ}}(u,v^{-1}).\]
\end{lemma}

\begin{proof}
We recall that, if $\Po$ is a parabolic subgroup of $\Go$ and $\Lo$ is a Levi complement of $\Po$ such that $\L'=N_{\G'}(\Lo,\Po)$, then, naming $\P'=N_{\G'}(\Po)$, we have that $\P'=\L'\ltimes\U'$ is a Levi decomposition, with $\U'=\operatorname{R}_u(\Po)$ (in $\G'$). We have a similar Levi decomposition for the centralizers of the element $s$, by \cite[proposition 1.11]{DiMi} (see section \ref{sec_gen_setting}): $\tensor*[^h]{}{}\P'(s)^{\circ}=\tensor*[^h]{}{}\L'(s)^{\circ} \ltimes \tensor*[^h]{}{}\U'(s)^{\circ}$. We will call 
$\tensor*[^h]{}{}\U'(s)^{\circ}=\tilde{\U}$. 
We have that $Q_{\tensor*[^h]{}{}\L'(s)^{\circ}}^{\G'(s)^{\circ}}(u,v^{-1})=\Trace((u,v^{-1})|H_c^*(\Y_{\tilde{\U}}))$, being $\Y_{\tilde{\U}}= \{\tilde g\tilde{\U} \in \G'(s)^{\circ}/\tilde{\U}~|~\tilde{g}^{-1}F(\tilde{g})\in\tilde{\U}F(\tilde{\U})\}$. 

%
%
%
%
The map $\phi=\ad(\dot{\tau}^{-1})$ is an isomorphism of algebraic groups from $\G'(s)^{\circ}$ to $\phi(\G'(s)^{\circ})$, that sends the decomposition $\tensor*[^h]{}{}\P'(s)^{\circ} = \tensor*[^h]{}{}\L'(s)^{\circ} \ltimes \tilde{\U}$ to $\phi(\tensor*[^h]{}{}\P'(s)^{\circ}) = \phi(\tensor*[^h]{}{}\L'(s)^{\circ}) \ltimes \phi(\tilde{\U})$. Furthermore
:
\[
\begin{array}{rcl}
\Y_{\phi(\tilde{\U})}& = &\{\tilde c\phi(\tilde{\U})\in \phi(\G'(s)^{\circ}) / \phi(\tilde{\U}) ~|~\tilde c^{-1}F(\tilde c)\in \phi(\tilde{\U})F(\phi(\tilde{\U}))\} \\
& = & \{\phi(c)\phi(\tilde{\U})~|~ c\in \G'(s)^{\circ},~\phi(c^{-1}F(c))\in \phi(\tilde{\U}F(\tilde{\U}))\}\\
& = & \{\phi(c \tilde{\U})~|~ c \tilde{\U} \in \G'(s)^{\circ}/\tilde{\U},~c^{-1}F(c)\in \tilde{\U}F(\tilde{\U})\},
\end{array}
\]
where we have used that $\phi=\ad(\dot{\tau}^{-1})$ commutes with $F$ since $\dot{\tau}\in\G^F$.
%
%
%
%

This means that the morphism of varieties
 $\phi$ restricts to a bijection (or, equivalently, a surjection with fibers isomorphic to points) from $\Y_{\tilde{\U}}$ to $\Y_{\phi(\tilde{\U})}$. Calling $(u\cdot,\cdot v^{-1})$ the action by left multiplication by $u\in\G'(s)^{\circ}$ and right multiplication by $v^{-1}\in\tensor*[^h]{}{}\L'(s)^{\circ}$ on $\Y_{\tilde{\U}}$, one can easily check that $\phi\circ (u\cdot,\cdot v^{-1})= (\phi(u)\cdot,\cdot \phi(v^{-1})) \circ \phi$. And both actions by multiplication are of finite order since $u$, $v$ are of finite order and the left multiplication commutes with the right multiplication. Hence \cite[proposition 10.12 (ii)]{DiMi2} yields
\[
\Trace((u,v^{-1})|H_c^*(\Y_{\tilde{\U}}))=\Trace((\phi(u),\phi(v)^{-1})|H_c^*(\Y_{\phi(\tilde{\U})})),
\]
as required.
\end{proof}

To finish, we see that 
\[
\tensor*[^{\dot{\tau}^{-1}h}]{}{}\chi(\phi(sv))= \chi(h^{-1}\dot{\tau}\dot{\tau}^{-1}(sv)\dot{\tau}\dot{\tau}^{-1}h)= \chi(h^{-1}svh)=\tensor*[^{h}]{}{}\chi(sv).
\]
Thus
\[
R_{\L'}^{\G}\chi(g)= 
\frac{1}{|\L'^F||\G'(s)^{\circ F}|}
\sum_{\tau\in\G^F/\G'^F}
\sum_{\{h\in\tau|s\in\tensor*[^h]{}{}\L'\}} 
\sum_{v\in (\tensor*[^h]{}{}\L'(s)^{\circ})_u^F}
Q_{\tensor*[^h]{}{}\L'(s)^{\circ}}^{\G'(s)^{\circ}}(u,v^{-1})
\tensor*[^h]{}{}\chi(sv),
\] and we have the result by merging the first two sums, and taking into account that $\G(s)^{\circ}= (\G^{\circ}(s))^{\circ} = \G'(s)^{\circ} $. 
\end{proof}

Now, we will introduce a function that will allow us to go from the Lusztig induction on disconnected groups to that of connected groups: Let $s$ be a semisimple element of $\G$. We define $d_s^{\G}$ as the map from the class functions of $\G^F$ to the class functions of $\G(s)^{\circ F}$ sending $\chi$ to
\[
d_s^{\G}(\chi):u\longmapsto\left\{ \begin{array}{ll}
\chi(su) & \text{if } u \text{ is unipotent,} \\
0 & \text{otherwise.}
\end{array} \right.
\] Here, again, we use that the quotient $\G/\Go$ consists on semisimple elements.

\begin{prop}[Exchange formulae]\label{Exch}
Let $\G'$, $\L'$ be as in the previous proposition. We have that:
\[
d_s^{\G}\circ R_{\L'}^{\G}= \frac{1}{|\L'^F||\G(s)^{\circ F}|}\sum_{\{h\in\G^F|s\in\tensor*[^h]{}{}\L'\}}|\tensor*[^h]{}{}\L'(s)^{\circ F}| R_{\tensor*[^h]{}{}\L'(s)^{\circ}}^{\G(s)^{\circ}} \circ d_s^{\tensor*[^h]{}{}\L'} \circ 
\ad(h),
\]
\[
d_t^{\L'} \circ \tensor*[^*]{}{}R_{\L'}^{\G}= \tensor*[^*]{}{}R_{\L'(t)^{\circ}}^{\G(t)^{\circ}}\circ d_t^{\G}.
\] 
\end{prop}
\begin{proof}
Let $u\in (\G(s)^{\circ})_u^F$, and let $\chi$ be a class function of $\L'$. On the one hand, we use \ref{char} on the left hand side:

\begin{align*}
d_s^{\G}(R_{\L'}^{\G} \chi)(u) & = R_{\L'}^{\G}\chi (su) \\
& = \frac{1}{|\L'^F||\G(s)^{\circ F}|}
\sum_{\{h\in\G^F|s\in\tensor*[^h]{}{}\L'\}}
\sum_{v\in (\tensor*[^h]{}{}\L'(s)^{\circ})_u^F}
Q_{\tensor*[^h]{}{}\L'(s)^{\circ}}^{\G(s)^{\circ}}(u,v^{-1})
\tensor*[^{h}]{}{}\chi(sv).
\end{align*}

On the other hand, we apply the definition of the map $R_{\tensor*[^h]{}{}\L'(s)^{\circ}}^{\G(s)^{\circ}}$ to have:
\[
\frac{1}{|\L'^F||\G(s)^{\circ F}|}\sum_{\{h\in\G^F|s\in\tensor*[^h]{}{}\L'\}}|\tensor*[^h]{}{}\L'(s)^{\circ F}| R_{\tensor*[^h]{}{}\L'(s)^{\circ}}^{\G(s)^{\circ}} \circ d_s^{\tensor*[^h]{}{}\L'} \circ 
\ad(h)(\chi) (u) 
\]
\[
= \frac{1}{|\L'^F||\G(s)^{\circ F}|}
\sum_{\{h\in\G^F|s\in\tensor*[^h]{}{}\L'\}}
\sum_{l\in \tensor*[^h]{}{}\L'(s)^{\circ F}} 
 \operatorname{Trace}((u,l^{-1})| H_c^*(\Y_{\tilde{\U}})) d_s^{\tensor*[^h]{}{}\L'} 
 \tensor*[^{h}]{}{}\chi(l)
\]
\[
= \frac{1}{|\L'^F||\G(s)^{\circ F}|}
\sum_{\{h\in\G^F|s\in\tensor*[^h]{}{}\L'\}}
\sum_{v\in (\tensor*[^h]{}{}\L'(s)^{\circ})_u^F} 
 \operatorname{Trace}((u,v^{-1})| H_c^*(\Y_{\tilde{\U}})) 
\tensor*[^{h}]{}{}\chi(sv),
\] where $\Y_{\tilde{\U}}$ is the usual Deligne-Lusztig variety for $\tensor*[^h]{}{}\L'(s)^{\circ} $ and $\G(s)^{\circ}$. For the second equality, we have used that $d_s^{\tensor*[^h]{}{}\L'}$ is zero in every non-unipotent element. And the last trace is $Q_{\tensor*[^h]{}{}\L'(s)^{\circ}}^{\G(s)^{\circ}}(u,v^{-1})$, so we obtain the required equality.

For the restriction, we take $\psi$ a class function of $\G^F$ and $v\in (\L'(t)^{\circ})_u^{F}$. We have:
\[ d_t^{\L'} \circ \tensor*[^*]{}{}R_{\L'}^{\G}\psi (v)=\tensor*[^*]{}{}R_{\L'}^{\G}(\psi)(tv)=
|\G(t)^{\circ F}|^{-1} \sum_{u\in (\G(t)^{\circ})_u^F} Q_{\L'(t)^{\circ}}^{\G(t)^{\circ}}(u^{-1},v)\psi(tu).
\]
And if $\Y_{\tilde{\U}'}$ is the usual Deligne-Lusztig variety for $\L'(t)^{\circ} $ and $\G(t)^{\circ}$:
\begin{align*}
\tensor*[^*]{}{}R_{\L'(t)^{\circ}}^{\G(t)^{\circ}}d_t^{\G}(\psi)(v) &= |\G(t)^{\circ F}|^{-1}\sum_{g\in \G(t)^{\circ F}}  \operatorname{Trace}((g,v)| H_c^*(\Y_{\tilde{\U}'})) d_t^{\G}(\psi)(g^{-1}) \\
& = |\G(t)^{\circ F}|^{-1}\sum_{u\in (\G(t)^{\circ})_u^F} 
\operatorname{Trace}((u^{-1},v)| H_c^*(\Y_{\tilde{\U}'})) \psi(tu).
\end{align*}
In the last equality, we have changed $g$ by $g^{-1}$ and we have used that $d_t^{\G}\psi$ vanishes on non-unipotent elements, whence the result.
\end{proof}

\section{General Mackey formula.}



Let us recall again the Mackey formula for connected reductive groups: Given two rational Levi subgroups $\Lo$ and $\Mo$ of $\Go$, 
we have that
\begin{equation}\label{Mackey_conn}
\tensor*[^*]{}{}R_{\Lo}^{\Go}\circ R_{\Mo}^{\Go}=\sum_{x\in \LoF\backslash \mathcal S_{\Go}(\Lo,\Mo)^F/ \MoF} R_{\Lo\cap\tensor*[^x]{}{}\Mo}^{\Lo} \circ \tensor*[^*]{}{}R_{\Lo\cap\tensor*[^x]{}{}\Mo}^{\tensor*[^x]{}{}\Mo} \circ \ad(x),
\end{equation}
where $\mathcal S_{\Go}(\Lo,\Mo)=\{x\in\Go~|~\Lo\cap\tensor*[^x]{}{}\Mo\text{ contains a maximal torus of }\Go\}$, so that $\Lo\cap\tensor*[^x]{}{}\Mo$ is a Levi subgroup of $\Go$. 
This formula holds when $\Go$, $\Mo$ and $\Lo$ satisfy one of the following conditions (see \cite[theorem 3.9]{BonMi}): 
\begin{itemize}
\item $\Lo$ and $\Mo$ are contained in rational parabolic subgroups. 
\item $\Lo$ or $\Mo$ is a torus. 
\item $q>2$ or $\Go$ has no component of type $\tensor*[^2]{}{}\operatorname{E}_6$, $\operatorname{E}_7$ or $\operatorname{E}_8$. 
\end{itemize}

%
%
%
%

We go back to our group $\G$. Let $\Po$, $\Qo$ be two parabolic subgroups of $\Go$, with respective Levi complements $\Lo$, $\Mo$. We consider the Levi subgroups $\L=N_{\G}(\Lo,\Po)$ and $\M=N_{\G}(\Mo,\Qo)$ of $\G$. We suppose that $\L$ and $\M$ are $F$-stable. We define
\[
\mathcal S_{\G}(\L,\M)=\{x\in\G|\L\cap\tensor*[^x]{}{}\M\text{ contains a maximal torus of } \Go\}.
\] It is the same definition than that for connected reductive groups. 
 But $x$ being in $\mathcal S_{\G}(\L,\M)$ does not imply that $\L\cap\tensor*[^x]{}{}\M$ is a Levi of $\G$. However, it will be it for a smaller group:

\begin{prop}\label{LxM_Levi}
Let $\pi$ be the quotient map $\G\longrightarrow\G/\Go$. Let $\G_x$ be the preimage by $\pi$ of the image of $\L\cap\tensor*[^x]{}{}\M$ by $\pi$. 
Then $\L\cap\tensor*[^x]{}{}\M$ is a Levi subgroup of $\G_x$.
\end{prop}
\begin{proof}
First of all, since the unipotent radical of $\G_x$ is contained in $\Go$ (it is connected, unipotent, closed and normal in $\Go$), it is a reductive group.
%
%
%
%
Let $\L_x=\L\cap\G_x$ and $\M_x=\M\cap\G_x$. These are Levis of $\G_x$: We have that $\G_x\cap N_{\G}(\Lo,\Po)= N_{\G_x}(\Lo,\Po)$, similarly for $\M_x$. We remark that $\L_x\cap\tensor*[^x]{}{}\M_x=\L\cap\tensor*[^x]{}{}\M$. 
Since $\pi(\G_x)=\G_x/\Go$ is the image of the first map mentioned in the proof, we can take representatives of $\G_x/\Go$ in $\L_x\cap\tensor*[^x]{}{}\M_x$. This implies that the map $\L_x\cap\tensor*[^x]{}{}\M_x\hookrightarrow \G_x\twoheadrightarrow \G_x/\Go$ is surjective, and its kernel is $\L_x\cap\tensor*[^x]{}{}\M_x\cap\Go=\Lo\cap\tensor*[^x]{}{}\Mo$ (see the remark after \cite[definition 1.4]{DiMi}). 
Thus $(\L_x\cap\tensor*[^x]{}{}\M_x)/(\Lo\cap\tensor*[^x]{}{}\Mo)\cong \G_x/\Go$. 

Therefore, we will suppose, without loss of generality, that $\G=\G_x$, $\L=\L_x$ and $\M=\M_x$, so that 
\[(\L\cap\tensor*[^x]{}{}\M)/(\Lo\cap\tensor*[^x]{}{}\Mo)\cong \G/\Go,\]
and now we have to prove that $\L\cap\tensor*[^x]{}{}\M$ is a Levi of $\G$.

We recall that $\Lo\cap\tensor*[^x]{}{}\Mo$ is a Levi complement of the parabolic subgroup $\Po\cap\tensor*[^x]{}{}\Qo\cdot \U$ of $\Go$,
being $\U$ the unipotent radical of $\Po$. Let $\mathbf Z=N_{\G}(\Lo\cap\tensor*[^x]{}{}\Mo, \Po\cap\tensor*[^x]{}{}\Qo\cdot \U)$. We have to prove that $\L\cap\tensor*[^x]{}{}\M=\mathbf Z$: 
\begin{itemize}[label=$\bullet$]
\item On the one hand, the inclusion to the right comes from the fact that, if $n\in\G$ normalizes $\Lo, \Mo, \Po$ and $\tensor*[^x]{}{}\Qo$, then it will normalize the intersections $\Lo\cap\tensor*[^x]{}{}\Mo$ and $\Po\cap\tensor*[^x]{}{}\Qo$. And, since $\tensor*[^n]{}{}\Po=\Po$ and $\tensor*[^n]{}{}\Lo=\Lo$, we have that $\Po=\Lo\ltimes\tensor*[^n]{}{}\U$, being $\tensor*[^n]{}{}\U$ a maximal unipotent closed connected normal subgroup of $\P$. Therefore, it is equal to $\U$.
\item On the other hand, we can construct, as at the beginning of the proof, a map $\mathbf Z\hookrightarrow \G\twoheadrightarrow \G/\Go$. It is surjective, since $\L\cap\tensor*[^x]{}{}\M\subset\mathbf Z$, and its kernel is $\mathbf Z \cap \Go$. Thus we have that $\mathbf Z/(\mathbf Z\cap\Go)\cong \G/\Go\cong (\L\cap\tensor*[^x]{}{}\M)/(\Lo\cap\tensor*[^x]{}{}\Mo)$. We can construct similarly, by using the inclusion $\L\cap\tensor*[^x]{}{}\M\hookrightarrow \mathbf Z$, an injection $  (\L\cap\tensor*[^x]{}{}\M)/(\Lo\cap\tensor*[^x]{}{}\Mo)\hookrightarrow \mathbf Z/(\mathbf Z\cap\Go)$. But $\mathbf Z\cap\Go =\Lo\cap\tensor*[^x]{}{}\Mo$. Furthermore, the two previous quotients are bijective finite sets, so the injection must be a bijection of finite sets. So, given any element $z\in\mathbf Z$, the class $z(\Lo\cap\tensor*[^x]{}{}\Mo)$ contains a unique coset $m(\Lo\cap\tensor*[^x]{}{}\Mo)$, with $m\in \L\cap\tensor*[^x]{}{}\M$. This cosets must be equal, hence $z\in m(\Lo\cap\tensor*[^x]{}{}\Mo)\subset \L\cap\tensor*[^x]{}{}\M$, we get the other inclusion.
%
\end{itemize} 

Therefore, $\L\cap\tensor*[^x]{}{}\M=\mathbf Z$ is a Levi subgroup of $\G_x$.
\end{proof}

For an $x\in\mathcal S_{\G}(\L,\M)$ we will have, following the previous section:
\[
R_{\L\cap\tensor*[^x]{}{}\M}^{\L}= \Ind_{\L_x^F}^{\L^F}\circ R_{\L\cap\tensor*[^x]{}{}\M}^{\L_x}, \tensor*[^*]{}{}R_{\L\cap\tensor*[^x]{}{}\M}^{\tensor*[^x]{}{}\M}= \tensor*[^*]{}{}R_{\L\cap \tensor*[^x]{}{}\M}^{\tensor*[^x]{}{}\M_x}\circ \Res_{\tensor*[^x]{}{}\M_x}^{\tensor*[^x]{}{}\M}.
\]

\begin{definition}
We will call \emph{Mackey formula} the following expression:
\[
\tensor*[^*]{}{}R_{\L}^{\G}\circ R_{\M}^{\G}=\sum_{x\in \L^F\backslash \mathcal S_{\G}(\L,\M)^F/ \M^F} R_{\L\cap\tensor*[^x]{}{}\M}^{\L} \circ \tensor*[^*]{}{}R_{\L\cap\tensor*[^x]{}{}\M}^{\tensor*[^x]{}{}\M} \circ \ad(x).
\]
\end{definition}

We can rewrite the Mackey formula as follows:
\[
\tensor*[^*]{}{}R_{\L}^{\G}\circ R_{\M}^{\G}=\sum_{x\in \mathcal S_{\G}(\L,\M)^F} \frac{|\L^F \cap \tensor*[^x]{}{}\M^F|}{|\L^F||\M^F|} R_{\L\cap\tensor*[^x]{}{}\M}^{\L} \circ \tensor*[^*]{}{}R_{\L\cap\tensor*[^x]{}{}\M}^{\tensor*[^x]{}{}\M} \circ \ad(x):
\]
If we change $x$ by $xm$ with $m\in \M^F$ the maps will not change; and the class functions of $\M^F$ are stable under conjugation by elements of $\M^F$. And if we change $x$ by $lx$ with $l\in\L^F$ we will have that $\L\cap \tensor*[^{lx}]{}{}\M=\tensor*[^l]{}{}(\L\cap \tensor*[^{x}]{}{}\M)$. We use then that $\tensor*[^*]{}{}R_{\L\cap\tensor*[^{lx}]{}{}\M}^{\tensor*[^{lx}]{}{}\M} \circ \ad(l)=\ad(l) \circ \tensor*[^*]{}{}R_{\L\cap\tensor*[^x]{}{}\M}^{\tensor*[^x]{}{}\M}$ and $R_{\L\cap\tensor*[^{lx}]{}{}\M}^{\L} \circ \ad(l) = \ad(l) \circ R_{\L\cap\tensor*[^{x}]{}{}\M}^{\L}$. And $\ad(l)$ is trivial on the class functions of $\L^F$. So we will have $\L^F\tensor*[^x]{}{}\M^F$ different elements giving us the same map. Therefore, for the formula, we should divide by the cardinal of this set, that is $\frac{|\L^F||\M^F|}{|\L^F \cap \tensor*[^x]{}{}\M^F|}$.

We can now state our main theorem:

\begin{theorem}\label{Mackey_ds}
Let $(\G,\L,\M)$ be a triple of a reductive group and two Levi subgroups such that: 
\begin{enumerate}[label=(\roman*)]
\item For all $x\in\mathcal S_{\G}(\L,\M)$, the group $\G_x$ defined in \ref{LxM_Levi} is normal in $\G$.
\item One of this conditions is satisfied:
\begin{enumerate}
\item $\L$ and $\M$ are contained in stable parabolic subgroups,
\item $\L$ or $\M$ come from a torus of $\Go$, 
\item $q>2$ or $\Go$ has no simple component 
of type $\tensor*[^2]{}{}\operatorname{E}_6$, $\operatorname{E}_7$ or $\operatorname{E}_8$.
\end{enumerate}
\item $\G/\Go$ consists on semisimple elements.
\end{enumerate}
Let $s$ be a semisimple element in $\L^F$. Then we have that:
\[
d_s^{\L} \circ \tensor*[^*]{}{}R_{\L}^{\G}\circ R_{\M}^{\G}
= \sum_{x\in \L^F\backslash \mathcal S_{\G}(\L,\M)^F/ \M^F}
d_s^{\L} \circ R_{\L\cap\tensor*[^x]{}{}\M}^{\L} 
\circ \tensor*[^*]{}{}R_{\L\cap\tensor*[^x]{}{}\M}^{\tensor*[^x]{}{}\M} \circ \ad(x).
\]
\end{theorem}

\begin{proof}
We will follow the proof of \cite[lemma 2.3.4.]{Bon}. We begin by exchanging $d_s^{\L}$ with the Lusztig maps on the first term of the formula, using the proposition \ref{Exch}. We obtain the sum
\[
\frac{1}{|\M^F||\G(s)^{\circ F}|}\sum_{\{g\in\G^F|s\in\tensor*[^g]{}{}\M\}}|\tensor*[^g]{}{}\M(s)^{\circ F}|\tensor*[^*]{}{}R_{\L(s)^{\circ}}^{\G(s)^{\circ}}\circ R_{\tensor*[^g]{}{}\M(s)^{\circ}}^{\G(s)^{\circ}}\circ d_s^{\tensor*[^g]{}{}\M}\circ \ad(g).
\]

\begin{lemma}
The Mackey formula is true for the Levi subgroups $\tensor*[^g]{}{}\M(s)^{\circ}$ and $\L(s)^{\circ}$ of $\G(s)^{\circ}$.
\end{lemma}

\begin{proof}
We recall that $\tensor*[^g]{}{}\M(s)^{\circ}$ and $\L(s)^{\circ}$ are $F$-stable Levi subgroups of the connected reductive group $\G(s)^{\circ}$.
Let us go through all conditions in (ii): 
\begin{itemize}
\item If (a) holds, then the centralizers of $s$ in the corresponding parabolic subgroups are also $F$-stable. This means that $\tensor*[^g]{}{}\M(s)^{\circ}$ and $\L(s)^{\circ}$ are split Levi subgroups in $\G(s)^{\circ}$, thus they satisfy \ref{Mackey_conn}, the Mackey formula for connected groups. 
\item If (b) holds and we suppose that $\L$ is the centralizer of a maximal torus of $\G$ and a Borel subgroup containing the torus, then from \cite[theorem 1.8]{DiMi} we have that $\L(s)^{\circ}$ is a maximal torus of $\G(s)^{\circ}$. This fulfills one of the conditions for the formula \ref{Mackey_conn}. 
\item If $q>2$, we apply again the Mackey formula for connected reductive groups. If $\Go$ is not of type $\tensor*[^2]{}{}\operatorname{E}_6$, $\operatorname{E}_7$ or $\operatorname{E}_8$, then neither will $\G(s)^{\circ}=(\G^{\circ}(s))^{\circ}$ be of that type, and we can apply, one last time, the Mackey formula for connected groups. 
\end{itemize}
\end{proof}

Using the previous lemma, we apply the Mackey formula for the centralizers of $s$. We will obtain intersections of the type $ \L(s)^{\circ} \cap \tensor*[^{hg}]{}{}\M(s)^{\circ}$
for $h\in \mathcal S_{\G(s)^{\circ}}(\L(s)^{\circ},\tensor*[^g]{}{}\M(s)^{\circ})^F$. We have that $ \L(s)^{\circ} \cap \tensor*[^{hg}]{}{}\M(s)^{\circ}= (\L \cap \tensor*[^{hg}]{}{}\M)(s)^{\circ}$: First of all, $\L(s) \cap \tensor*[^{hg}]{}{}\M(s)= (\L \cap \tensor*[^{hg}]{}{}\M)(s)$. Then, it is clear that $(\L \cap \tensor*[^{hg}]{}{}\M)(s)^{\circ}$ is contained in both $\L(s)^{\circ} $ and $\tensor*[^{hg}]{}{}\M(s)^{\circ}$. Furthermore, by the definition of $h$, $\L(s)^{\circ} \cap \tensor*[^{hg}]{}{}\M(s)^{\circ}$ is a Levi subgroup of $\G(s)^{\circ}$, as it is the intersection of two Levi subgroups containing a common maximal torus. Therefore, it is a connected group on $(\L \cap \tensor*[^{hg}]{}{}\M)(s)$, and the other inclusion holds.

We obtain, using the alternative form of the Mackey formula: 
\begin{align} \label{dsRLG}
d_s^{\L} \circ \tensor*[^*]{}{}R_{\L}^{\G}\circ R_{\M}^{\G} &= 
\sum_{\{g\in \G^F|s\in\tensor*[^g]{}{}\M\}}
\sum_{h\in \mathcal S_{\G(s)^{\circ}}(\L(s)^{\circ},\tensor*[^g]{}{}\M(s)^{\circ})^F}
\frac{|(\L \cap \tensor*[^{hg}]{}{}\M)(s)^{\circ F}|}{|\M^F| |\G(s)^{\circ F}| |\L(s)^{\circ F}|} \nonumber\\
 & \times R_{(\L \cap \tensor*[^{hg}]{}{}\M)(s)^{\circ}}^{\L(s)^{\circ}}  
 \circ \tensor*[^*]{}{}R_{(\L \cap \tensor*[^{hg}]{}{}\M)(s)^{\circ}}^{\tensor*[^{hg}]{}{}\M(s)^{\circ}}
 \circ d_s^{\tensor*[^{hg}]{}{}\M} \circ \ad (hg).
\end{align}

For the other term, we can use the exchange formula in \ref{Exch}, due to the conditions (i) and (iii) on $\G$. We obtain:
\[
\sum_{x\in \L^F\backslash \mathcal S_{\G}(\L,\M)^F/ \M^F}
d_s^{\L} \circ R_{\L\cap\tensor*[^x]{}{}\M}^{\L} 
\circ \tensor*[^*]{}{}R_{\L\cap\tensor*[^x]{}{}\M}^{\tensor*[^x]{}{}\M} \circ \ad(x) 
\]
\[
= \sum_{x\in \mathcal S_{\G}(\L,\M)^F} \frac{|\L^F \cap \tensor*[^x]{}{}\M^F|}{|\L^F||\M^F|}
d_s^{\L} \circ R_{\L\cap\tensor*[^x]{}{}\M}^{\L} 
\circ \tensor*[^*]{}{}R_{\L\cap\tensor*[^x]{}{}\M}^{\tensor*[^x]{}{}\M} \circ \ad(x)
\]
\[
= \sum_{x\in \mathcal S_{\G}(\L,\M)^F}\frac{1}{|\L^F||\M^F||\L(s)^{\circ F}|}
\sum_{\{l\in\L^F| s\in\L\cap\tensor*[^{lx}]{}{}\M\}}
|(\L\cap\tensor*[^{lx}]{}{}\M)(s)^{\circ F}|
\]
\[
 \times R_{(\L \cap \tensor*[^{lx}]{}{}\M)(s)^{\circ}}^{\L(s)^{\circ}}  
 \circ \tensor*[^*]{}{}R_{(\L \cap \tensor*[^{lx}]{}{}\M)(s)^{\circ}}^{\tensor*[^{lx}]{}{}\M(s)^{\circ}}
 \circ d_s^{\tensor*[^{lx}]{}{}\M} \circ \ad(lx)
\](the exchange with the first map will add a factor $|\L^F \cap \tensor*[^x]{}{}\M^F|$ on the denominator). 
Next, the map
\[\begin{array}{rcl}
\{(l,x)\in \L^F\times \mathcal S_{\G}(\L,\M)^F| s\in\L\cap\tensor*[^{lx}]{}{}M\} & \longrightarrow & \{y\in\mathcal S_{\G}(\L,\M)^F | s\in\L\cap\tensor*[^{y}]{}{}\M\} \\
(l,x) & \longmapsto & lx
\end{array}
\]
is clearly surjective, and the fiber of each element $y\in\mathcal S_{\G}(\L,\M)^F$ is $\{(l,l^{-1}y)|l\in\L^F\}$, hence of cardinal $|\L^F|$. Therefore, the sum above is equal to
\[
\sum_{\{y\in\mathcal S_{\G}(\L,\M)^F | s\in\L\cap\tensor*[^{y}]{}{}\M\}}
\frac{|(\L\cap\tensor*[^{y}]{}{}\M)(s)^{\circ F}|}{|\M^F||\L(s)^{\circ F}|}
  R_{(\L \cap \tensor*[^{y}]{}{}\M)(s)^{\circ}}^{\L(s)^{\circ}}  
 \circ \tensor*[^*]{}{}R_{(\L \cap \tensor*[^{y}]{}{}\M)(s)^{\circ}}^{\tensor*[^{y}]{}{}\M(s)^{\circ}}
 \circ d_s^{\tensor*[^{y}]{}{}\M} \circ \ad (y),
\]

Finally, we repeat the process with the sum \ref{dsRLG}: We consider the map from
$\{(g,h)\in\G^F \times \G(s)^{\circ F} | s\in\tensor*[^{g}]{}{}\M,
 h\in \mathcal S_{\G(s)^{\circ}}(\L(s)^{\circ},\tensor*[^g]{}{}\M(s)^{\circ})\} $ to $\{y\in\mathcal S_{\G}(\L,\M)^F | s\in\L\cap\tensor*[^{y}]{}{}\M\}$ that sends $(g,h)$ to $hg$. It is:
\begin{itemize}
\item Well defined: 
By the definition of $h$, $\L(s)^{\circ}$ and $\tensor*[^{hg}]{}{}\M(s)^{\circ}$ contain a common maximal torus $\mathbf S$ of $\G(s)^{\circ}$. From (iv) of \cite[theorem 1.8.]{DiMi}, the centralizers $C_{\Go}(\mathbf S)$, $C_{\Lo}(\mathbf S)$ and $C_{\tensor*[^{hg}]{}{}\Mo}(\mathbf S)$ are maximal tori of $\Go$, $\Lo$ and $\tensor*[^{hg}]{}{}\Mo$, respectively. $\Lo$ and $\tensor*[^{hg}]{}{}\Mo$ are reductive groups of the same rank than $\Go$, what implies that the three tori are maximal on $\Go$. Furthermore, the last two are contained in $C_{\G^{\circ}}(\mathbf S)$, thus $C_{\Lo}(\mathbf S)= C_{\Go}(\mathbf S) =C_{\tensor*[^{hg}]{}{}\Mo}(\mathbf S)$. Therefore, $C_{\Go}(\mathbf S)\subset \L\cap\tensor*[^{hg}]{}{}\M$ and then $hg\in S_{\G}(\L,\M)^F$. We already had that $s=\tensor*[^h]{}{}s\in\tensor*[^{hg}]{}{}\M$.


\item Surjective: Let $y\in\mathcal S_{\G}(\L,\M)^F$ such that $ s\in\L\cap\tensor*[^{y}]{}{}\M$. We have that $s\in\tensor*[^{y}]{}{}\M$, we just need to verify that $\L(s)^{\circ}\cap \tensor*[^{y}]{}{}\M(s)^{\circ}$ contains a maximal torus of $\G(s)^{\circ}$. We recall that $\L\cap\tensor*[^{y}]{}{}\M$ is a Levi of a certain $\G'$. We restrict ourselves to $\G'$: Since $s$ is semisimple in the reductive group $\L\cap\tensor*[^{y}]{}{}\M$, by \cite[proposition 1.3]{DiMi}, there exists a "torus" $\T'$ of this group containing $s$. 
Its connected component $\To$ is a maximal torus of $\Go$, $\Lo$ and $\tensor*[^{y}]{}{}\Mo$, 
and it is $s$-stable, thus from (iii) of \cite[theorem 1.8.]{DiMi} we get that $(\To(s))^{\circ}$ is a maximal torus of $(\Go(s))^{\circ}$, $(\Lo(s))^{\circ}$ and $(\tensor*[^{y}]{}{}\Mo(s))^{\circ}$. Therefore, $(y,1)$ belongs to the preimage of $y$, since $(\Lo(s))^{\circ}=\L(s)^{\circ}$ and 
$(\tensor*[^{y}]{}{}\Mo(s))^{\circ}=\tensor*[^{y}]{}{}\M(s)^{\circ}$.
%
%

\item The fiber of $y$ is $\{(hy,h^{-1})|h\in \G(s)^{\circ F}\}$ so it is always of cardinal $|\G(s)^{\circ F}|$.
\end{itemize}

So the sum \ref{dsRLG} is equal to the last sum, whence the result.

\end{proof}

Since the equality of the theorem holds for every semisimple element $s$ of $\L^F$, and every element of $\G$ admits a Jordan decomposition, we obtain our main result:

\begin{corollary}[Mackey formula]\label{main_corollary}
Let $\G$ be a reductive algebraic group. Let $\L$ and $\M$ be rational Levi subgroups of $\G$. Let $\mathcal S_{\G}(\L,\M)$ be the set of $x\in\G$ such that $\L\cap\tensor*[^x]{}{}\M$ contains a maximal torus of $\Go$. We suppose that $\G$, $\L$ and $\M$ satisfy the conditions in \ref{Mackey_ds}. Then:
\[
\tensor*[^*]{}{}R_{\L}^{\G}\circ R_{\M}^{\G}
= \sum_{x\in \L^F\backslash \mathcal S_{\G}(\L,\M)^F/ \M^F}
R_{\L\cap\tensor*[^x]{}{}\M}^{\L} 
\circ \tensor*[^*]{}{}R_{\L\cap\tensor*[^x]{}{}\M}^{\tensor*[^x]{}{}\M} \circ \ad(x).
\]
\end{corollary}

\bibliography{phd_biblio_mackey}

\bibliographystyle{alpha}

\end{document}